\documentclass[graybox]{svmult}

\usepackage{mathptmx}       
\usepackage{helvet}         
\usepackage{courier}        
\usepackage{type1cm}        
\usepackage{makeidx}         
\usepackage{graphicx}        
\usepackage{multicol}        
\usepackage[bottom]{footmisc}

\usepackage{amsmath,amsfonts,amssymb,mathtools}

\makeindex             

\newtheorem{assumption}{Assumption}
\newtheorem{prop}{Proposition}[section]

\newcommand{\na}{\nabla}
\newcommand{\pa}{\partial}
\newcommand{\eps}{\varepsilon}
\newcommand{\om}{\omega}
\newcommand{\Om}{\Omega}
\newcommand{\De}{\Delta}
\newcommand{\si}{\sigma}

\newcommand{\IOm}{I\times \Om}
\newcommand{\ImOm}{I_m\times \Om}

\newcommand{\jump}[1] {\mbox{$[\![ #1 ]\!]$}}

\newcommand{\norm}[1]{\lVert#1\rVert}
\newcommand{\ltwonorm}[1]{\lVert#1\rVert_{L^2(\Omega)}}

\newcommand{\abs}[1]{\lvert#1\rvert}
\newcommand{\lh}{|\ln h|}
\newcommand{\Ppol}[1]{\mathcal{P}_{#1}}

\newcommand{\sgn}{\operatorname{sgn}}
\newcommand{\half}{\frac{1}{2}}

\newcommand{\ou}{\bar u}

\newcommand{\oq}{\bar q}
\newcommand{\oz}{\bar z}

\newcommand{\R}{\mathbb{R}}

\newcommand{\Qad}{Q_{\text{ad}}}
\newcommand{\Xk}{X^0_k}
\newcommand{\Xkh}{X^{0,1}_{k,h}}

\begin{document}

\title*{Optimal a priori error estimates of  parabolic optimal control problems with a moving point control}
\titlerunning{Parabolic optimal control problems with a moving point control}
\author{Dmitriy Leykekhman and Boris Vexler}
\institute{Dmitriy Leykekhman \at Department of Mathematics,
               University of Connecticut,
              Storrs,
              CT~06269, USA, \email{dmitriy.leykekhman@uconn.edu}
\and Boris Vexler \at Technical University of Munich,
Chair of Optimal Control,
Center for Mathematical Sciences,
Boltzmannstra{\ss}e 3,
85748 Garching by Munich, Germany, \email{vexler@ma.tum.de}}
%
%
\maketitle

\abstract{In this paper we consider a parabolic optimal control problem with a Dirac type control with moving point source in two space dimensions. We discretize the problem with piecewise constant functions in time and continuous piecewise linear finite elements in space. For this discretization we show optimal order of convergence with respect to the time and the space discretization parameters modulo some logarithmic terms. Error analysis for the same problem was carried out in the recent paper \cite{GongW_YanN_2016a}, however, the analysis  there contains a serious flaw. One of the main goals of this paper is to provide the correct proof. The main ingredients of our analysis  are the global and local error estimates on a curve, that have an independent interest.}

\section{Introduction}
In this paper we provide numerical analysis for the following optimal control problem:
\begin{equation}  \label{eq:intro-obj}
    \min_{q,u}J(q,u):= \frac{1}{2} \int_0^T  \|u(t) - \hat{u}(t)\|_{L^2(\Omega)}^2 dt
    + \frac{\alpha}{2}  \int_0^T  |q(t)|^2 dt
\end{equation}
subject to the second order parabolic equation
\begin{subequations}     \label{eq:intro-state}
\begin{align}
    u_t(t,x)-\Delta u(t,x) &= q(t)\delta_{\gamma(t)}, & (t,x) &\in \IOm,\;  \\
    u(t,x) &= 0,    & (t,x) &\in I\times\pa\Omega, \\
   u(0,x) &= 0,    & x &\in \Omega
\end{align}
\end{subequations}
and subject to pointwise control constraints
\begin{equation}\label{eq:control_constraints}
q_a \le q(t) \le q_b \quad \text{a.\,e. in } I.
\end{equation}
Here $I=(0,T)$, $\Omega\subset \mathbb{R}^2$ is a convex polygonal domain and $\delta_{\gamma(t)}$ is the Dirac delta function at point $x_t=\gamma(t)$ at each $t$. We will assume:
\begin{assumption}\label{assumptions gamma_1}
\begin{itemize}
\item $\gamma\in C^1(\bar{I})$ and $\max_{t\in \bar{I}}|\gamma'(t)|\le C_\gamma$.
\end{itemize}
\end{assumption}
\begin{assumption}\label{assumptions gamma_2}
\begin{itemize}
\item $\gamma(t)\subset \overline{\Om}_0\subset\subset \Om_1$, for any $t\in I$, with $\overline{\Om}_1\subset\subset \Om$.
\end{itemize}
\end{assumption}
The parameter $\alpha$ is assumed to be positive and the desired state $\hat u$ fulfills $\hat u \in L^2(I;L^\infty(\Omega))$. The control bounds $q_a,q_b \in \R\cup\{\pm \infty\}$ fulfill $q_a <q_b$. The precise functional-analytic setting is discussed in the next section.

For the discretization, we consider the standard continuous piecewise linear finite elements in space and piecewise constant discontinuous Galerkin method in time. This is a special case ($r=0$, $s=1$) of so called dG($r$)cG($s$) discretization, see e.g.~\cite{ErikssonK_JohnsonC_ThomeeV_1985} for the analysis of the method for parabolic problems and  e.g.~\cite{MeidnerD_VexlerB_2008a, MeidnerD_VexlerB_2008b} for error estimates in the context of optimal control problems. Throughout, we will denote by $h$ the spatial mesh size and by $k$ the size of time steps, see Section~\ref{sec: best approximation} for details.

 The main  result of the paper is the following.
\begin{theorem}\label{thm: main}
Let $\bar{q}$ be optimal control for the problem \eqref{eq:intro-obj}-\eqref{eq:intro-state} and $\bar{q}_{kh}$ be the optimal dG(0)cG(1) solution. Then there exists a constant $C$ independent of $h$ and $k$ such that
$$
\|\bar{q}-\bar{q}_{kh}\|_{L^2(I)} \le C\left(\lh^3(k+h^2)+C_\gamma\lh k\right)\left(\|\bar{q}\|_{L^2(I)}+\|\hat{u}\|_{L^2(I;L^\infty(\Om))}\right).
$$
\end{theorem}

We would also like to point out that in addition to the optimal order estimate, modulo logarithmic terms,   our analysis does not require any relationship between the sizes of the  space  discretization $h$ and the time steps $k$. 

The problem with fixed location of the point source  (i.e. with $\delta_{x_0}(x)$ for some fixed $x_0\in \Om$) starting with the work of Lions \cite{LionsJL_1971},  was investigated in a number of publications, see \cite{AmourouxM_BabaryJP_1978, BanksHT_1992, ChryssoverghiI_1981, DroniouJ_RaymondJP_2000, NguyenPA_RaymondJP_2011} for the continuous problem and \cite{GongW_HinzeM_ZhouZ_2014, LeykekhmanD_VexlerB_2013, LeykekhmanD_VexlerB_2016c} for the finite element approximation and error estimates. There is also a closely related problem of measured valued controls, which received a lot of attention lately   \cite{CasasE_ClasonC_KunischK_2013, CasasE_KunischK_2016, CasasE_VexlerB_ZuazuaE_2015, CasasE_ZuazuaE_2013, KunischK_PieperK_VexlerB_2014}.

The problem with moving Dirac  was considered in \cite{CastroC_ZuazuaE_2004a, NguyenPA_RaymondJP_2001}  on a continuous level. The error analysis was carried out in the recent paper \cite{GongW_YanN_2016a}. However, the analysis  there contains a serious flaw. The last inequality in the estimate $(3.33)$ in \cite{GongW_YanN_2016a} is not correct. One of the main goals of this paper is to provide the correct proof. The main ingredients of our analysis  are the global and local error estimates on a curve, Theorem \ref{thm: global best approx} and Theorem \ref{thm: apriori local}, respectively. These results are new and have an independent interest.

Throughout the paper we use the usual notation for Lebesgue and Sobolev spaces. We denote by $(\cdot,\cdot)_\Omega$ the inner product in $L^2(\Omega)$ and by $(\cdot,\cdot)_{\tilde{I} \times \Omega}$ the inner product in $L^2(\tilde{I}\times\Omega)$ for any subinterval $\tilde{I} \subset I$.

The rest of the paper is organized as follows. In Section \ref{sec: optimality and regularity} we discuss the functional analytic setting of the problem, state the optimality system and prove regularity results for the state and for the adjoint state. In Section \ref{sec: best approximation} we establish important global and local best approximation results along the curve for the heat equation. Finally in Section \ref{sec: optimal control} we prove our main result.

\section{Optimal control problem and regularity} \label{sec: optimality and regularity}
In order to state the functional analytic setting for the optimal control problem, we first introduce the auxiliary problem
\begin{equation}\label{eq: heat equation}
\begin{aligned}
v_t(t,x)-\Delta v(t,x) &= f(t,x), & (t,x) &\in \IOm,\;  \\
    v(t,x) &= 0,    & (t,x) &\in I\times\pa\Omega, \\
   v(0,x) &= 0,    & x &\in \Omega,
\end{aligned}
\end{equation}
with a right-hand side $f \in L^2(I;L^p(\Omega))$ for some $1<p<\infty$. This equation possesses a unique solution
\[
v \in L^2(I; H^1_0(\Omega))\cap H^1(I;H^{-1}(\Omega)).
\]
Due to the convexity of the polygonal domain $\Omega$ the solution $v$ possesses an additional regularity for $p=2$:
\[
v \in L^2(I;H^2(\Omega)\cap H^1_0(\Omega)) \cap H^1(I;L^2(\Omega)),
\]
with the corresponding estimate
\begin{equation}\label{eq: H2 regularity for v}
\|v\|_{L^2(I;H^2(\Omega))}+\|v_t\|_{L^2(I;L^2(\Omega))} \le C \|f\|_{L^2(I;L^2(\Omega))},
\end{equation}
 see, e.g.,~\cite{EvansLC_2010}. From  the Sobolev embedding $ H^2(\Om)\hookrightarrow W^{1,s}(\Om)$ for any $s<\infty$ in two space dimensions and  the previous lemma we can establish the following result  for $s>2$,
\begin{equation}\label{eq: Holder regularity for v}
\norm{v}_{L^2(I;W^{1,s}(\Omega))}\le Cs \norm{v}_{L^2(I;H^2(\Om))}\le Cs\norm{f}_{L^2(I;L^2(\Omega))}.
\end{equation}
The exact form of the constant can be traced, for example, from the
proof of \cite[Thm.~10.8]{AltHW_2016}.
In addition, there holds the following regularity result (see \cite{LeykekhmanD_VexlerB_2013}).
\begin{lemma}\label{lemma: regularity}
If $f\in L^2(I;L^p(\Omega))$ for an arbitrary $p>1$, then $v\in L^2(I;C(\Omega))$ and
$$
\|v\|_{L^2(I;C(\Omega))}\le C_p\|f\|_{L^2(I;L^p(\Omega))},
$$
where $C_p\sim \frac{1}{p-1}$, as $p\to 1$.
\end{lemma}

We will also need the following local regularity result (see \cite{LeykekhmanD_VexlerB_2013}). 
\begin{lemma}\label{lemma: regularity local_Om_0}
Let $\Om_0 \subset\subset \Om_1\subset\subset \Om$ and $f\in L^2(I;L^2(\Omega))\cap L^2(I;L^p(\Om_1))$ for some  $2 \le p < \infty$. Then $v \in L^2(I;W^{2,p}(\Om_0)) \cap H^1(I;L^p(\Om_0))$  and there exists a constant $C$ independent of $p$ such that
$$
\|v_t\|_{L^2(I;L^p(\Om_0))}+\| v\|_{L^2(I;W^{2,p}(\Om_0))}\le Cp(\|f\|_{L^2(I;L^p(\Om_1))}+\|f\|_{L^2(I;L^2(\Om))}).
$$
\end{lemma}

To introduce a weak solution of the state equation~\eqref{eq:intro-state} we use the method of transposition, (cf.~\cite{LionsJL_MagenesE_Vol2}). For a given control $q \in Q = L^2(I)$ we denote by $u = u(q) \in L^2(I;L^p(\Omega))$ with $2\le p<\infty$ a weak solution of~\eqref{eq:intro-state}, if for all $\varphi \in L^2(I;L^{p'}(\Omega))$ with $\frac{1}{p}+\frac{1}{p'}=1$ there holds
\[
\langle u,\varphi\rangle_{L^2(I;L^p(\Omega)),L^2(I;L^{p'}(\Omega))} = \int_I w(t,\gamma(t))q(t) \,dt,
\]
where $w \in L^2(I;W^{2,p'}(\Omega)\cap H^1_0(\Omega)) \cap H^1(I;L^{p'}(\Omega))$ is the weak solution of the adjoint equation
\begin{equation}\label{eq:adjoint_w}
\begin{aligned}
-w_t(t,x)-\Delta w(t,x) &= \varphi(t,x), & (t,x) &\in \IOm,\;  \\
    w(t,x) &= 0,    & (t,x) &\in I\times\pa\Omega, \\
   w(T,x) &= 0,    & x &\in \Omega.
\end{aligned}
\end{equation}
The existence of this weak solution $u=u(q)$ follows by duality using the embedding $L^2(I;W^{2,p'}(\Omega)) \hookrightarrow L^2(I;C(\Omega))$ for $p'>1$. Using Lemma~\ref{lemma: regularity} we can prove additional regularity for the state variable $u=u(q)$.

\begin{prop} \label{prop:reg_u}
Without lose of generality we assume $2\le p<\infty$. Let $q \in Q = L^2(I)$ be given and $u =u(q)$ be the solution of the state equation~\eqref{eq:intro-state}. Then $u\in L^2(I;L^p(\Omega))$ for any $p<\infty$ and the following estimate holds for $p\to \infty$ with a constant $C$ independent of $p$,
\[
\norm{u}_{L^2(I;L^p(\Omega))} \le C p \norm{q}_{L^2(I)}.
\]
\end{prop}
\begin{proof}
To establish the result we use a  duality argument. There holds
\[
\norm{u}_{L^2(I;L^p(\Om))}=\sup_{\norm{\varphi}_{L^2(I;L^{p'}(\Om))}=1} (u,\varphi)_{I\times \Omega}, \quad \text{where} \quad \frac{1}{p}+\frac{1}{p'}=1.
\]
Let $w$ be the solution to~\eqref{eq:adjoint_w} for $\varphi \in L^2(I;L^{p'}(\Om))$ with $\norm{\varphi}_{L^2(I;L^{p'}(\Om))}=1$.
From Lemma~\ref{lemma: regularity}, $w\in L^2(I;C(\Omega))$ and the following estimate holds
$$
\|w\|_{L^2(I;C(\Omega))}\le\frac{C}{p'-1} \|\varphi\|_{L^2(I;L^{p'}(\Omega))}=\frac{C}{p'-1}\le Cp, \ \text{as}\ p\to \infty.
$$
Thus,
\[
\begin{aligned}
\|u\|_{L^2(I;L^p(\Om))}&=\sup_{\|\varphi\|_{L^2(I;L^{p'}(\Om))}=1}(u,\varphi)_{\IOm}\\
&=\int_I q(t)w(t,\gamma(t))\,dt\le \|q\|_{L^2(I)}\|w\|_{L^2(I;C(\Omega))}\le Cp\|q\|_{L^2(I)}.
\end{aligned}
\]
\end{proof}

\begin{remark}
We would like to note that the above regularity requires only Assumption 2 on $\gamma$. Higher regularity of $\gamma$ is needed for optimal order error estimates only.
\end{remark}

A  further regularity result for the state equation follows from~\cite{ElschnerJ_RehbergJ_SchmidtG_2007}.
\begin{prop}\label{prop:reg_u2}
Let $q \in Q = L^2(I)$ be given and $u =u(q)$ be the solution of the state equation~\eqref{eq:intro-state}. Then for each $1 < s < 2$ there holds
\[
u \in L^2(I;W^{1,s}_0(\Omega)) \quad\text{and}\quad  u_t \in L^2(I;W^{-1,s}(\Omega)).
\]
Moreover, the state $u$ fulfills the following weak formulation
\[
\langle u_t, \varphi \rangle + (\nabla u, \nabla \varphi) = \int_I q(t)\varphi(t,\gamma(t)) \,dt \quad \text{for all }\; \varphi \in L^2(I;W^{1,s'}_0(\Omega)),
\]
where $\frac{1}{s'}+\frac{1}{s} = 1$ and $\langle \cdot,\cdot\rangle$ is the duality product between $ L^2(I;W^{-1,s}(\Omega))$ and $ L^2(I;W^{1,s'}_0(\Omega))$.
\end{prop}
\begin{proof}
For $s<2$ we have $s'>2$ and therefore $W^{1,s'}_0(\Omega)$ is embedded into $C(\bar \Omega)$. Therefore the right-hand side $q(t)\delta_{\gamma(t)}$ of the state equation can be identified with an element in $L^2(I;W^{-1,s}(\Omega))$. Using the result from~\cite[Theorem 5.1]{ElschnerJ_RehbergJ_SchmidtG_2007} on maximal parabolic regularity and exploiting the fact that $-\Delta \colon W^{1,s}_0(\Omega) \to W^{-1,s}(\Omega)$ is an isomorphism, see~\cite{JerisonD_KenigCE_1995}, we obtain
\[
u \in L^2(I;W^{1,s}_0(\Omega)) \quad\text{and}\quad u_t \in L^2(I;W^{-1,s}(\Omega)).
\]
 Given the above regularity the corresponding weak formulation is fulfilled by a standard density argument.
\end{proof}

As the next step we introduce the reduced cost functional $j \colon Q \to \mathbb{R}$ on the control space $Q = L^2(I)$ by
\[
j(q) = J(q,u(q)),
\]
where $J$ is the cost function in~\eqref{eq:intro-obj} and $u(q)$ is the weak solution of the state equation~\eqref{eq:intro-state} as defined above. The optimal control problem can then be equivalently reformulated as
\begin{equation}\label{eq:min_pr_red}
\min\,  j(q), \quad q \in \Qad,
\end{equation}
where the set of admissible controls is defined according to~\eqref{eq:control_constraints} by
\begin{equation}\label{eq:qad}
\Qad = \{q \in Q\ |\ q_a \le q(t) \le q_b \;\text{a. e. in } I\}.
\end{equation}
By standard arguments this optimization problem possesses a unique solution $\oq \in Q = L^2(I)$ with the corresponding state $\ou = u(\oq) \in L^2(I;L^p(\Omega))$ for all $p<\infty$, see Proposition~\ref{prop:reg_u} for the regularity of $\ou$. Due to the fact, that this optimal control problem is convex, the solution $\oq$ is equivalently characterized by the optimality condition
\begin{equation}\label{eq:opt_cond}
j'(\oq)(\partial q - \oq) \ge 0 \quad \text{for all } \partial q \in \Qad.
\end{equation}
The (directional) derivative $j'(q)(\partial q)$ for given $q, \partial q \in Q$ can be expressed as
\[
j'(q)(\partial q) = \int_I \left(\alpha q(t) + z(t,\gamma(t)) \right) \partial q(t)\, dt,
\]
where $z = z(q)$ is the solution of the adjoint equation
\begin{subequations}     \label{eq:intro-adjoint}
\begin{align}
    -z_t(t,x)-\Delta z(t,x) &= u(t,x)-\hat{u}(t,x), & (t,x) &\in \IOm,\; \label{eq:intro-adjoint_a} \\
    z(t,x) &= 0,    & (t,x) &\in I\times\pa\Omega, \\
   z(T,x) &= 0,    & x &\in \Omega,
\end{align}
\end{subequations}
and $u=u(q)$ on the right-hand side of~\eqref{eq:intro-adjoint_a} is the solution of the state equation~\eqref{eq:intro-state}.
The adjoint solution, which corresponds to the optimal control $\oq$ is denoted by $\oz = z(\oq)$.

The optimality condition~\eqref{eq:opt_cond} is a variational inequality, which can be equivalently formulated using the  projection
\[
P_{\Qad} \colon Q \to \Qad, \quad P_{\Qad}(q)(t) = \min\bigl(q_b,\max(q_a,q(t))\bigr).
\]
The resulting condition reads:
\begin{equation}\label{eq:proj}
\oq(t) = P_{\Qad}\left(-\frac{1}{\alpha} \oz(t,\gamma(t))\right).
\end{equation}

In the next proposition we provide regularity results for the solution of the adjoint equation.
\begin{prop}\label{prop:reg_z}
Let $q\in Q$ be given, let $u=u(q)$ be the corresponding state fulfilling~\eqref{eq:intro-state} and let $z=z(q)$ be the corresponding adjoint state fulfilling~\eqref{eq:intro-adjoint}. Then,
\begin{itemize}
\item[(a)]\  $z \in L^2(I;H^2(\Omega)\cap H^1_0(\Omega)) \cap H^1(I;L^2(\Omega))$ and the following estimate holds
\[
\norm{\nabla^2 z}_{L^2(I;L^2(\Omega))} +  \norm{z_t}_{L^2(I;L^2(\Omega))} \le C (\norm{q}_{L^2(I)} + \norm{\hat u}_{L^2(I;L^2(\Omega))}).
\]

\item[(b)]\ If $\Om_0 \subset\subset \Om$, then $z \in L^2(I;W^{2,p}(\Om_0)) \cap H^1(I;L^p(\Om_0))$ for all $2 \le p < \infty$  and the following estimate holds
\[
\norm{\nabla^2 z}_{L^2(I;L^p(\Om_0))} +  \norm{z_t}_{L^2(I;L^p(\Om_0))} \le C p^2  (\norm{q}_{L^2(I)}+\norm{\hat u}_{L^2(I;L^\infty(\Omega))}).
\]
\end{itemize}
\end{prop}
\begin{proof}
\begin{itemize}
\item[(a)]\ The right-hand side of the adjoint equation fulfills $u - \hat u \in L^2(I;L^p(\Omega))$ for all $1<p<\infty$, see Proposition~\ref{prop:reg_u}. Due to the convexity of the domain $\Omega$ we directly obtain  $z \in L^2(I;H^2(\Omega)\cap H^1_0(\Omega)) \cap H^1(I;L^2(\Omega))$ and the estimate
\[
\norm{\nabla^2 z}_{L^2(I;L^2(\Omega))} +  \norm{z_t}_{L^2(I;L^2(\Omega))} \le C \norm{u - \hat u}_{L^2(I;L^2(\Omega))}.
\]
The result from Proposition~\ref{prop:reg_u} leads directly to the first estimate.

\item[(b)]\
From Lemma~\ref{lemma: regularity local_Om_0} for $p\geq 2$ we have
\[
\norm{\nabla^2 z}_{L^2(I;L^p(\Om_0))} +  \norm{z_t}_{L^2(I;L^p(\Om_0))} \le Cp\norm{u-\hat u}_{L^2(I;L^p(\Omega))}.
\]
Hence, by the triangle inequality and Proposition~\ref{prop:reg_u} we obtain
$$
\|u - \hat u\|_{L^2(I;L^p(\Omega))} \le C\left(p \norm{q}_{L^2(I)} + \norm{\hat u}_{L^2(I;L^\infty(\Omega))}\right).
$$
That completes the proof.
\end{itemize}
\end{proof}

\section{Discretization and the best approximation type results} \label{sec: best approximation}

\subsection{Space-time discretization and notation}
For  discretization of the problem under the consideration
 we introduce a partitions of $I =[0,T]$ into subintervals $I_m = (t_{m-1}, t_m]$ of length $k_m = t_m-t_{m-1}$, where $0 = t_0 < t_1 <\cdots < t_{M-1} < t_M =T$. We assume that 
 \begin{equation}\label{eq: quasi assumption time}
 k_{m+1}\le\kappa k_m,\quad m=1,\dots,M-1,\quad \text{for some}\quad \kappa>0. 
 \end{equation}
 The maximal time step is denoted by $k =\max_{m} k_m$.
The semidiscrete space $\Xk$ of piecewise constant functions in time is defined by
\[
\Xk=\{v_{k}\in L^2(I;H^1_0(\Om)) :\ v_{k}\rvert_{I_m}\in \Ppol{0}(I_m;H^1_0(\Om)), \ m=1,2,\dots,M \},
\]
where $\Ppol{0}(I;V)$ is the space of constant functions in time with values in Banach space $V$.
We will employ the following notation for functions in $\Xk$
\begin{equation}\label{def: time jumps}
v^+_m=\lim_{\eps\to 0^+}v(t_m+\eps):=v_{m+1}, \quad v^-_m=\lim_{\eps\to 0^+}v(t_m-\eps)=v(t_m):=v_m, \quad [v]_m=v^+_m-v^-_m.
\end{equation}

Let $\mathcal{T}$  denote  a quasi-uniform triangulation of $\Om$  with a mesh size $h$, i.e., $\mathcal{T} = \{\tau\}$ is a partition of $\Om$ into triangles $\tau$ of diameter $h_\tau$ such that for $h=\max_{\tau} h_\tau$,
$$
\operatorname{diam}(\tau)\le h \le C |\tau|^{\frac{1}{2}}, \quad \forall \tau\in \mathcal{T}
$$
hold. Let $V_h$ be the set of all functions in $H^1_0(\Om)$ that are linear on each $\tau$, i.e. $V_h$ is the usual space of continuous piecewise linear finite elements.
 We will require the modified Cl\'{e}ment interpolant $i_h \colon L^1(\Om) \to V_h$ and the $L^2$-projection $P_h \colon L^2(\Omega) \to V_h$ defined by
\begin{equation}\label{eq:l2_proj}
(P_hv,\chi)_{\Om} = (v,\chi)_{\Om}, \quad \forall \chi\in V_h.
\end{equation}
To obtain the fully discrete approximation we consider the space-time finite element space
\begin{equation} \label{def: space_time}
\Xkh=\{v_{kh}\in \Xk :\ v_{kh}|_{I_m}\in \Ppol{0}(I_m;V_h), \ m=1,2,\dots,M \}.
\end{equation}
We will also need the following semidiscrete projection $\pi_k \colon C(\bar I; H^1_0(\Omega))\to \Xk$
defined by
\begin{equation}\label{eq:pi_k}
\pi_k v\rvert_{I_m}=v(t_{m}), \quad m=1,2,\dots, M,
\end{equation}
and the fully discrete projection $\pi_{kh}\colon C(\bar I; L^1(\Omega))\to \Xkh$ defined by $\pi_{kh}=i_h\pi_k$.

To introduce the dG(0)cG(1) discretization we define the following bilinear form
\begin{equation}\label{eq: bilinear form B}
 B(v,\varphi)=\sum_{m=1}^M \langle v_t,\varphi \rangle_{I_m \times \Omega}+(\na v,\na \varphi)_{\IOm}+\sum_{m=2}^M([v]_{m-1},\varphi_{m-1}^+)_\Om+(v_{0}^+,\varphi_{0}^+)_\Om,
\end{equation}
where $\langle \cdot,\cdot \rangle_{I_m \times \Omega}$ is the duality product between $ L^2(I_m;W^{-1,s}(\Omega))$ and $ L^2(I_m;W^{1,s'}_0(\Omega))$. We note, that the first sum vanishes for $v \in \Xk$.
Rearranging the terms, we obtain an equivalent (dual) expression for $B$:
\begin{equation}\label{eq:B_dual}
 B(v,\varphi)= - \sum_{m=1}^M \langle v,\varphi_t \rangle_{I_m \times \Omega}+ (\na v,\na \varphi)_{\IOm}-\sum_{m=1}^{M-1} (v_m^-,[\varphi_k]_m)_\Om + (v_M^-,\varphi_M^-)_\Om.
\end{equation}

In the two following theorems  we establish global and local best approximation type results along the curve for the  error between the solution $v$ of the auxiliary equation~\eqref{eq: heat equation} and its dG(0)cG(1) approximation $v_{kh} \in \Xkh$ defined as
\begin{equation}\label{eq: discrete heat with RHS}
B(v_{kh},\varphi_{kh})=(f,\varphi_{kh})_{\IOm} \quad \text{for all }\; \varphi_{kh}\in \Xkh.
\end{equation}
Since dG(0)cG(1) method is a consistent discretization we have the following Galerkin orthogonality relation:
\[
B(v-v_{kh},\varphi_{kh}) = 0 \quad \text{for all }\; \varphi_{kh}\in \Xkh.
\]

\subsection{Discretization of the curve and the weight function}

To define fully discrete optimization problem we will also require a discretization of the curve $\gamma$. We define $\gamma_k=\pi_k\gamma$ by
\begin{equation}\label{eq: discrete gamma}
    \gamma_k \rvert_{I_m}=\gamma(t_m):=\gamma_{k,m}\in \Om_0, \quad  m=1,2,\dots,M,
\end{equation}
i.e., $\gamma_k$ is a piecewise constant approximation of $\gamma$.
Next we introduce a weight function
\begin{equation}\label{eq: weight_sigma}
\sigma(t,x) = \sqrt{|x-\gamma(t)|^2+h^2}
\end{equation}
and a discrete piecewise constant in time approximation
\begin{equation}\label{eq: weight_sigma_k}
\sigma_k(t,x) = \sqrt{|x-\gamma_k(t)|^2+h^2}.
\end{equation}
Define
\begin{equation}\label{eq: weight_sigma_km}
\sigma_{k,m}:=\sigma_{k}\rvert_{I_m}=\sigma_k(t_m,x)=\sigma(t_m,x).
\end{equation}
One can easily check that $\sigma$ and $\sigma_k$ satisfy the following properties for any $(t,x)\in I\times \Om$,
\begin{subequations}
\begin{align}
\|\sigma^{-1}(t,\cdot)\|_{L^2(\Om)},\|\sigma_k^{-1}(t,\cdot)\|_{L^2(\Om)}&\le C\lh^{\frac{1}{2}},\quad t\in \bar{I}, \label{eq: property 1 of sigma_k}\\
|\na \sigma(t,x)|,\ |\na \sigma_k(t,x)|&\le C, \label{eq: property 2 of sigma_k}\\
|\na^2 \sigma_k(t,x)|&\le C| \sigma_k^{-1}(t,x)|,\label{eq: property 3 of sigma_k}\\
|\sigma_t(t,x)|&\le | \na\sigma(t,x)|\cdot|\gamma'(t)|\le CC_\gamma,\label{eq: property 4 of sigma}\\
\max_{x\in\tau}{\sigma}(x,t)&\le C\min_{x\in\tau}{\sigma}(x,t), \quad \forall \tau\in \mathcal{T}.\label{eq: property 5 of sigma}
\end{align}
\end{subequations}

\subsection{Global error estimate along the curve}
In this section we prove the following global approximation result.
\begin{theorem}[Global best approximation] \label{thm: global best approx}
Assume $v$ and $v_{kh}$ satisfy \eqref{eq: heat equation} and \eqref{eq: discrete heat with RHS}
respectively. Then there exists a constant $C$ independent of $k$ and $h$ such that for any $1\le p\le \infty$,
\begin{align*}
\int_I|(v-v_{kh})(t,\gamma_k(t))|^2dt &\le C \lh^2\times\\ 
&\inf_{\chi \in \Xkh} \left( \norm{v-\chi}^2_{L^2(I;L^\infty(\Omega))}  + h^{-\frac{4}{p}} \norm{\pi_k v-\chi}^2_{L^2(I;L^p(\Omega))} \right).
\end{align*}
\end{theorem}
\begin{proof}
To establish the result we use a duality argument. First, we introduce a smoothed Delta function, which we will denote by $\tilde{\delta}_{\gamma_k}$. This function on each $I_m$ is defined as $\tilde{\delta}_{\gamma_{k,m}}$ and supported in one cell, which we denote by $\tau^0_m$, i.e.
$$
(\chi, \tilde{\delta}_{\gamma_{k,m}})_{\tau^0_m}=\chi(\gamma_{k,m})=\chi(\gamma(t_m)), \quad \forall \chi\in \mathbb{P}^1(\tau^0_m),\quad m=1,2,\dots,M.
$$
In addition we also have (see \cite[~Appendix]{SchatzAH_WahlbinLB_1995})
\begin{equation}\label{delta1}
 \|\tilde{\delta}_{\gamma_{k}}\|_{W^s_p(\Om)} \le C h^{-s-2(1-\frac{1}{p})}, \quad 1\le p \le \infty, \quad s=0,1.
\end{equation}
Thus in particular $\|\tilde{\delta}_{\gamma_k}\|_{L^1(\Om)} \le C$, $\|\tilde{\delta}_{\gamma_k}\|_{L^2(\Om)} \le Ch^{-1}$, and  $\|\tilde{\delta}_{\gamma_k}\|_{L^\infty(\Om)} \le Ch^{-2}$.

We define $g$ to be a solution to the following backward parabolic problem
\begin{equation}\label{eq: heat with dirac RHS}
\begin{aligned}
-g_t(t,x)-\Delta g(t,x) &= v_{kh}(t,\gamma_k(t))\tilde{\delta}_{\gamma_k}(x),& (t,x)\in \IOm,\;  \\
g(t,x)&=0, &(t,x) \in I\times\pa\Omega, \\
g(T,x)&=0, & x \in \Omega.
\end{aligned}
\end{equation}
There holds
$$
\begin{aligned}
\int_{I\times\Om}v_{kh}(t,\gamma_k(t))\tilde{\delta}_{\gamma_k}(x)\varphi_{kh}(t,x)dtdx &= \sum_{m=1}^M\int_{I_m}v_{kh}(t,\gamma_k(t))\left(\int_{\Om}\tilde{\delta}_{\gamma_k}(x)\varphi_{kh}(t,x)dx\right)dt\\
&=\sum_{m=1}^M\int_{I_m}v_{kh}(t,\gamma_k(t))\varphi_{kh}(t,\gamma_k(t))dt\\
&=\int_{I}v_{kh}(t,\gamma_k(t))\varphi_{kh}(t,\gamma_k(t))dt.
\end{aligned}
$$
Let $g_{kh} \in \Xkh$ be dG(0)cG(1) solution defined by
\begin{equation}\label{eq: discrete heat with dirac RHS}
B(\varphi_{kh}, g_{kh})=(v_{kh}(t,\gamma_k(t))\tilde{\delta}_{\gamma_k},\varphi_{kh})_{\IOm},\quad \forall \varphi_{kh}\in \Xkh.
\end{equation}
Then using that dG(0)cG(1) method is consistent, we have
\begin{equation}\label{eq: starting with B}
\begin{aligned}
\int_0^T|v_{kh}(t,\gamma_k(t))|^2dt &= B(v_{kh}, g_{kh})=B(v, g_{kh})\\
&=(\na  v, \na g_{kh})_{\IOm}-\sum_{m=1}^{M}(v_m,[g_{kh}]_m)_{\Om},
 \end{aligned}
\end{equation}
where we have used the dual expression \eqref{eq:B_dual} for the bilinear form $B$ and the fact that the last term in~\eqref{eq:B_dual} can be included in the sum by setting $g_{kh,M+1} = 0$ and defining consequently $[g_{kh}]_M = -g_{kh,M}$. The first sum in~\eqref{eq:B_dual} vanishes due to $g_{kh} \in \Xkh$.
For each $t$, integrating by parts elementwise and using that $g_{kh}$ is linear in the spacial variable, by the H\"{o}lder's inequality we have
\begin{equation}\label{eq: before jumps}
(\na  v, \na g_{kh})_{\Om}=\frac{1}{2}\sum_{\tau}(v, \jump{\pa_n g_{kh}})_{\pa\tau}\le C\|v\|_{L^\infty(\Om)}\sum_{\tau}\|\jump{\pa_n g_{kh}}\|_{L^1(\pa\tau)},
\end{equation}
where $\jump{\pa_n g_{kh}}$ denotes the jumps of the normal derivatives across the element faces.

From Lemma 2.4 in \cite{RannacherR_1991a} we have
$$
\sum_{\tau}\|\jump{\pa_n g_{kh}}\|_{L^1(\pa\tau)}\le C\lh^{\frac{1}{2}}\left(\|\sigma_k\Delta_h g_{kh}\|_{L^2(\Om)}+\|\na g_{kh}\|_{L^2(\Om)}\right),
$$
where  $\Delta_h \colon V_h \to V_h$ is the discrete Laplace operator, defined by
$$
-(\Delta_h v_h,\chi)_\Omega = (\na v_h,\na \chi)_\Omega,\quad \forall \chi\in V_h.
$$
To estimate the term involving the jumps in~\eqref{eq: starting with B}, we  first use the H\"{o}lder's inequality and the inverse estimate to obtain
\begin{equation}\label{eq: global jumps}
\sum_{m=1}^{M}(v_m,[g_{kh}]_m)_{\Om}\le C \sum_{m=1}^{M} k_m^{\frac{1}{2}}\|v_m\|_{L^p(\Om)}k_m^{-\frac{1}{2}}h^{-\frac{2}{p}}\|[g_{kh}]_m\|_{L^1(\Om)}.
 \end{equation}
Now we use the fact that the equation~\eqref{eq: discrete heat with dirac RHS} can be rewritten on the each time level as
$$
(\na  \varphi_{kh}, \na g_{kh})_{\ImOm}-(\varphi_{kh,m},[g_{kh}]_m)_{\Om}=(v_{kh}(t,\gamma_k(t))\tilde{\delta}_{\gamma_k},\varphi_{kh})_{\ImOm},
$$
or equivalently as
\begin{equation}\label{eq:gkh_m}
-k_m\Delta_h g_{kh,m}-[g_{kh}]_m=k_m v_{kh,m}(\gamma_{k,m})P_h\tilde{\delta}_{\gamma_{k,m}},
\end{equation}
where $P_h$ is the $L^2$-projection, see~\eqref{eq:l2_proj}.
From \eqref{eq:gkh_m} by the triangle inequality, we obtain
$$
\|[g_{kh}]_m\|_{L^1(\Om)}\le k_m\|\Delta_h g_{kh,m}\|_{L^1(\Om)}+k_m\|P_h\tilde{\delta}_{\gamma_{k,m}}\|_{L^1(\Om)} \abs{v_{kh,m}(\gamma_{k,m})}.
$$
Using that the $L^2$-projection is stable in $L^1$-norm (cf. \cite{CrouzeixM_ThomeeV_1987a}), we have
$$
\|P_h\tilde{\delta}_{\gamma_{k,m}}\|_{L^1(\Om)}\le C\|\tilde{\delta}_{\gamma_{k,m}}\|_{L^1(\Om)}\le C.
$$
Inserting the above estimate into \eqref{eq: global jumps} and using \eqref{eq: property 1 of sigma_k}, we obtain
\begin{align*}
&\sum_{m=1}^M(v_m,[g_{kh}]_m)_{\Om}\le Ch^{-\frac{2}{p}}\sum_{m=1}^M k_m^{\frac{1}{2}}\|v_m\|_{L^p(\Om)}k_m^{\frac{1}{2}}\left(\|\Delta_h g_{kh,m}\|_{L^1(\Om)}+|v_{kh,m}(\gamma_{k,m})|\right)\\
&\le Ch^{-\frac{2}{p}}\left(\sum_{m=1}^M k_m\|v_m\|^2_{L^p(\Om)}\right)^{\frac{1}{2}}\left(\sum_{m=1}^Mk_m\|\Delta_h g_{kh,m}\|^2_{L^1(\Om)}+k_m|v_{kh,m}(\gamma_{k,m})|^2\right)^{\frac{1}{2}}\\
&\le Ch^{-\frac{2}{p}}\|\pi_k v\|_{L^2(I; L^p(\Om))}\left(\int_0^T\lh\|\sigma_k\Delta_h g_{kh}\|^2_{L^2(\Om)}+|v_{kh}(t,\gamma_{k}(t))|^2dt\right)^{\frac{1}{2}}.
 \end{align*}
Combining \eqref{eq: starting with B} and \eqref{eq: before jumps} with the above estimates we have
\begin{equation}\label{eq: u_kh plus others}
\begin{aligned}
&\int_0^T|v_{kh}(t,\gamma_{k}(t))|^2dt\le C\lh^{\frac{1}{2}}\left(\|v\|_{L^2(I;L^\infty(\Om))}+h^{-\frac{2}{p}}\|\pi_k v\|_{L^2(I; L^p(\Om))}\right)\times\\
&\left(\int_0^T\|\sigma_k\Delta_h g_{kh}\|^2_{L^2(\Om)}+\|\na g_{kh}\|^2_{L^2(\Om)}+|v_{kh}(t,\gamma_{k}(t))|^2dt\right)^{\frac{1}{2}}.
 \end{aligned}
 \end{equation}
To complete the proof of the theorem it is sufficient to show
\begin{equation}\label{eq: needed estimate for greens}
\int_0^T\left(\|\sigma_k\Delta_h g_{kh}\|^2_{L^2(\Om)}+\|\na g_{kh}\|^2_{L^2(\Om)}\right)dt\le C|\ln{h}|\int_0^T|v_{kh}(t,\gamma_{k}(t))|^2dt.
\end{equation}
Then from \eqref{eq: u_kh plus others} and \eqref {eq: needed estimate for greens} it would follow that
$$
\int_0^T|v_{kh}(t,\gamma_{k}(t))|^2dt\le C\lh^2\left(\|v\|^2_{L^2(I;L^\infty(\Om))}+h^{-\frac{4}{p}}\|\pi_k v\|^2_{L^2(I; L^p(\Om))}\right).
$$
Then using that the dG(0)cG(1) method is invariant on $\Xkh$,  by replacing $v$ an $v_{kh}$ with $v-\chi$ and $v_{kh}-\chi$ for any $\chi\in Xkh$, we obtain Theorem \ref{thm: global best approx}.

The estimate \eqref{eq: needed estimate for greens} will follow from the series of lemmas. The first lemma treats the term $\norm{\sigma_k\Delta_h g_{kh}}^2_{L^2(I;L^2(\Om))}$.
\begin{lemma}\label{lemma 1}
For any $\eps>0$ there exists $C_\eps$ such that
\begin{align*}
\int_0^T \|\sigma_k\Delta_h g_{kh}\|^2_{L^2(\Om)}dt &\le C_\eps \int_0^T\left(|v_{kh}(t,\gamma_{k}(t))|^2+\|\na g_{kh}\|^2_{L^2(\Om)}\right)dt\\
&\quad+\eps\sum_{m=1}^M k_m^{-1} \|\sigma_{k,m}[g_{kh}]_m\|^2_{L^2(\Om)},
\end{align*}
where  $\sigma_{k}$ and $\sigma_{k,m}$ are defined  in \eqref{eq: weight_sigma_k} and \eqref{eq: weight_sigma_km}, respectively.
\end{lemma}
\begin{proof}
The equation \eqref{eq: discrete heat with dirac RHS} for each time interval $I_m$ can be rewritten as~\eqref{eq:gkh_m}.
Multiplying~\eqref{eq:gkh_m} with $\varphi=-\sigma^2_{k}\Delta_h g_{kh}$ and integrating over $I_m\times \Om$, we have
\begin{align*}
\int_{I_m}&\|\sigma_{k,m}\Delta_h g_{kh}\|^2_{L^2(\Om)}dt\\
&=-([g_{kh}]_m,\sigma^2_{k,m}\Delta_h g_{kh,m})_\Om-(v_{kh}(t,\gamma_{k,m})P_h\tilde{\delta}_{\gamma_k},\sigma^2_{k,m}\Delta_h g_{kh})_{\ImOm}\\
&=-(P_h(\sigma^2_{k,m}[g_{kh}]_m),\Delta_h g_{kh,m})_\Om-(v_{kh}(t,\gamma_{k,m})P_h\tilde{\delta}_{\gamma_k},\sigma^2_{k,m}\Delta_h g_{kh})_{\ImOm}\\
&=(\na (\sigma^2_{k,m}[g_{kh}]_m),\na g_{kh,m})_\Om+(\na(P_h-I)(\sigma^2_{k,m}[g_{kh}]_m),\na g_{kh,m})_\Om\\
&\quad -(v_{kh}(t,\gamma_{k,m})P_h\tilde{\delta}_{\gamma_k},\sigma^2_{k,m}\Delta_h g_{kh})_{\ImOm}\\
&=J_1+J_2+J_3.
\end{align*}
We have
\begin{align*}
J_1 &= 2(\sigma_{k,m}\na\sigma_{k,m} [g_{kh}]_m,\na g_{kh,m})_\Om + (\sigma_{k,m} [\na g_{kh}]_m,\sigma_{k,m} \na g_{kh,m})_\Om=J_{11}+J_{12}.
\end{align*}
By the Cauchy-Schwarz inequality and using \eqref{eq: property 2 of sigma_k} we get
$$
J_{11}\le C\|\sigma_{k,m} [g_{kh}]_m\|_{L^2(\Om)}\|\na g_{kh,m}\|_{L^2(\Om)}.
$$
On the other hand we have
\begin{align*}
J_{12}&=([\sigma_{k}\na g_{kh}]_m,\sigma_{k,m} \na g_{kh,m})_\Om+((\sigma_{k,m}-\sigma_{k,m+1})\na g_{kh,m+1},\sigma_{k,m} \na g_{kh,m})_\Om\\&=J_{121}+J_{122}.
\end{align*}
Using the identity
\begin{equation}\label{eq: jumps identity}
([w_{kh}]_m,w_{kh,m})_\Om=\frac{1}{2}\| w_{kh,m+1}\|^2_{L^2(\Om)}-\frac{1}{2}\|w_{kh,m}\|^2_{L^2(\Om)}-\frac{1}{2}\| [w_{kh}]_m\|^2_{L^2(\Om)},
\end{equation}
we have
$$
J_{121}=\frac{1}{2}\|\sigma_{k,m+1} \na g_{kh,m+1}\|^2_{L^2(\Om)}-\frac{1}{2}\|\sigma_{k,m} \na g_{kh,m}\|^2_{L^2(\Om)}-\frac{1}{2}\| [\sigma_k\na g_{kh}]_m\|^2_{L^2(\Om)}.
$$
By the Cauchy-Schwarz inequality, we obtain
$$
\begin{aligned}
J_{122}&\le \|(\sigma_{k,m}-\sigma_{k,m+1})\na g_{kh,m+1}\|_{L^2(\Om)}\|\sigma_{k,m} \na g_{kh,m}\|_{L^2(\Om)}\\
&\le CC_\gamma k_m\|\na g_{kh,m+1}\|_{L^2(\Om)}\|\sigma_{k,m} \na g_{kh,m}\|_{L^2(\Om)},
\end{aligned}
$$
where in the last step we used that from \eqref{eq: property 4 of sigma}
$$
|\sigma_{k,m}(x)-\sigma_{k,m+1}(x)|=|\sigma(t_m,x)-\sigma(t_{m+1},x)|\le Ck_m| \sigma_t(\tilde{t},x)|\le CC_\gamma k_m,
$$
for some $\tilde{t}\in I_m$.
Using the Young's inequality for $J_{11}$, neglecting  $-\frac{1}{2}\| [\sigma_k\na g_{kh}]_m\|^2_{L^2(\Om)}$, and using the assumption on the time steps $k_m\le \kappa k_{m+1}$ and that $\si_k\le C$, we obtain
\begin{equation}\label{eq: J_1 thm Global error}
\begin{aligned}
 J_1 &\le \frac{1}{2}\| \sigma_{k,m+1} \na g_{kh,m+1}\|^2_{L^2(\Om)}-\frac{1}{2}\|\sigma_{k,m} \na g_{kh,m}\|^2_{L^2(\Om)} +\frac{\eps}{k_m} \|\sigma_{k,m} [g_{kh}]_m\|^2_{L^2(\Om)}\\
&\quad+ C_\eps k_m\|\na g_{kh,m}\|^2_{L^2(\Om)}+Ck_{m+1}\|\na g_{kh,m+1}\|^2_{L^2(\Om)}.
\end{aligned}
\end{equation}
To estimate $J_2$, first by the Cauchy-Schwarz inequality and the approximation theory we have
$$
J_2= \sum_\tau (\na(P_h-I)(\sigma^2_{k,m}[g_{kh}]_m),\na g_{kh,m})_\tau
\le Ch\sum_\tau\|\na^2(\sigma^2_{k,m}[g_{kh}]_m)\|_{L^2(\tau)}\|\na g_{kh,m}\|_{L^2(\tau)}.
$$
Using that $g_{kh}$ is piecewise linear we have
$$
\na^2(\sigma^2[g_{kh}]_m)= \na^2(\sigma^2)[g_{kh}]_m+\na(\sigma^2)\cdot\na [g_{kh}]_m \quad \text{on }\; \tau.
$$
There holds $\pa_{ij}(\sigma^2)=2(\pa_i \sigma)(\pa_j \sigma)+2\sigma\pa_{ij}\sigma$ and $\na(\sigma^2) = 2 \sigma \nabla \sigma$. Thus by the properties of $\sigma$~\eqref{eq: property 2 of sigma_k} and~\eqref{eq: property 3 of sigma_k},
we have
\[
\abs{\na^2(\sigma^2)} \le C \quad\text{and}\quad \abs{\na(\sigma^2)} \le C\, \sigma.
\]
Same estimates hold for $\sigma_k$.
Using these estimates, the fact that $h\le \sigma_k$ and the inverse inequality (in view of \eqref{eq: property 5 of sigma} the inverise inequality is valid with $\sigma$ inside the norm), we obtain
\begin{equation}\label{eq: estimates for J_2}
\begin{aligned}
J_2&\le C\sum_\tau\left(h\|[g_{kh}]_m\|_{L^2(\tau)}+h\|\sigma_{k,m}\na [g_{kh}]_m\|_{L^2(\tau)}\right)\|\na g_{kh,m}\|_{L^2(\tau)}\\
& \le C\sum_\tau\left(\|\sigma_{k,m}[g_{kh}]_m\|_{L^2(\tau)}+C_{inv}\|\sigma_{k,m} [g_{kh}]_m\|_{L^2(\tau)}\right)\|\na g_{kh,m}\|_{L^2(\tau)}\\
&\le  C\sum_\tau\|\sigma_{k,m}[g_{kh}]_m\|_{L^2(\tau)}\|\na g_{kh,m}\|_{L^2(\tau)}\\
&\le C_\eps k_m\|\na g_{kh,m}\|^2_{L^2(\Om)}+\frac{\eps}{k_m}\|\sigma_{k,m} [g_{kh}]_m\|^2_{L^2(\Om)}.
\end{aligned}
\end{equation}
To estimate $J_3$ we first notice that
\begin{equation}\label{eq: estimates sigmaP_hdelta}
\|\sigma_k P_h\tilde{\delta}_{\gamma_k}\|_{L^2(\Om)}\le C.
\end{equation}
The proof is identical to the proof  of $(3.21)$ in \cite{LeykekhmanD_VexlerB_2013}.

 By the Cauchy-Schwarz inequality, \eqref{eq: estimates sigmaP_hdelta}, and the Young's inequality, we obtain
\begin{equation}\label{eq: estimates for J_3}
J_3\le C\int_{I_m}|v_{kh}(t,\gamma_k)|^2dt+\frac{1}{2}\int_{I_m}\|\sigma_{k,m}\Delta_h g_{kh,m}\|^2_{L^2(\Om)}dt.
\end{equation}
Using the  estimates \eqref{eq: J_1 thm Global error}, \eqref{eq: estimates for J_2}, and \eqref{eq: estimates for J_3}  we have
$$
\begin{aligned}
\int_{I_m}\|\sigma_{k,m}\Delta_h g_{kh}\|^2_{L^2(\Om)}dt &\le C_\eps\int_{I_m}\left(|v_{kh}(t,\gamma_k(t))|^2+\|\na g_{kh}\|^2_{L^2(\Om)}\right)dt\\
&+CC_\gamma\int_{I_{m+1}}\|\na g_{kh}\|^2_{L^2(\Om)}dt+\frac{\eps}{k_m}\|\sigma_{k,m}[g_{kh}]_m\|^2_{L^2(\Om)}\\
&+\frac{1}{2}\|\sigma_{k,m+1} \na g_{kh,m+1}\|^2_{L^2(\Om)}-\frac{1}{2}\|\sigma_{k,m} \na g_{kh,m}\|^2_{L^2(\Om)}.
\end{aligned}
$$
Summing over $m$ and using that $g_{kh,M+1}=0$ we obtain the lemma.
\end{proof}

The second lemma treats the term involving jumps.
\begin{lemma}\label{lemma 2}
There exists a constant $C$ such that
$$
\sum_{m=1}^M k_m^{-1} \|\sigma_{k,m}[g_{kh}]_m\|^2_{L^2(\Om)}\le C\int_0^T\left(\|\sigma_k\Delta_h g_{kh}\|^2_{L^2(\Om)}+|v_{kh}(t,\gamma_k(t)|^2\right)dt.
$$
\end{lemma}
\begin{proof}
We test~\eqref{eq:gkh_m} with $\varphi\rvert_{I_m}=\sigma^2_{k,m}[g_{kh}]_m$ and obtain
\begin{equation}\label{eq: terms in Lemma 2}
\begin{aligned}
\|\sigma_{k,m}[g_{kh}]_m\|^2_{L^2(\Om)} =& -(\Delta_h g_{kh},\sigma^2_{k,m}[g_{kh}]_m)_{\ImOm}\\
&-(v_{kh}(t,\gamma_k(t))P_h\tilde{\delta}_{\gamma_k},\sigma^2_{k,m}[g_{kh}]_m)_{\ImOm}.
\end{aligned}
\end{equation}
The first term on the right hand side of \eqref{eq: terms in Lemma 2} using the Young's inequality can be  estimated as
$$
(\Delta_h g_{kh},\sigma^2_{k,m}[g_{kh}]_m)_{\ImOm}\le Ck_m\int_{I_m}\|\sigma_{k,m}\Delta_h g_{kh}\|^2_{L^2(\Om)}dt+\frac{1}{4}\|\sigma_{k,m}[g_{kh}]_m\|^2_{L^2(\Om)}.
$$
The last term on the right hand side of \eqref{eq: terms in Lemma 2} can easily be estimated using~\eqref{eq: estimates sigmaP_hdelta} as
$$
(v_{kh}(t,\gamma_{k,m})P_h\tilde{\delta}_{\gamma_k},\sigma^2_{k,m}[g_{kh}]_m)_{\ImOm}\le Ck_m\int_{I_m}|v_{kh}(t,\gamma_k(t))|^2dt+\frac{1}{4}\|\sigma_{k,m}[g_{kh}]_m\|^2_{L^2(\Om)}.
$$
Combining the above two estimates we obtain
$$
\|\sigma_{k,m}[g_{kh}]_m\|^2_{L^2(\Om)}\le Ck_m \int_{I_m}\left( \|\sigma_{k,m}\Delta_h g_{kh}\|^2_{L^2(\Om)}+|v_{kh}(t,\gamma_k(t))|^2\right)dt.
$$
Summing over $m$ we obtain the lemma.
\end{proof}

\begin{lemma}\label{lemma 3}
There exists a constant $C$ such that
$$
\norm{\na g_{kh}}^2_{L^2(I\times\Om)} \le C\lh\int_0^T |v_{kh}(t,\gamma_k(t))|^2dt.
$$
\end{lemma}
\begin{proof}
Adding the primal~\eqref{eq: bilinear form B} and the dual~\eqref{eq:B_dual} representation of the bilinear form $B(\cdot,\cdot)$ one immediately arrives at
\[
\norm{\nabla v}^2_{L^2(I\times\Om)} \le B(v,v) \quad \text{for all }\; v \in \Xk,
\]
see e.g.,~\cite{MeidnerD_VexlerB_2008a}. Applying this inequality together with the discrete Sobolev inequality, see~\cite[~Lemma 4.9.2]{BrennerSC_ScottLR_2008}, results in
\[
\begin{aligned}
\norm{\nabla g_{kh}}^2_{L^2(I\times\Om)} &\le B(g_{kh},g_{kh}) = (v_{kh}(t,\gamma_k(t))\tilde{\delta}_{\gamma_k},g_{kh})_{\IOm}\\& = \int_0^T v_{kh}(t,\gamma_k(t)) g_{kh}(t,\gamma_k(t)) \, dt\\
&\le \left(\int_0^T \abs{v_{kh}(t,\gamma_k(t))}^2\,dt\right)^{\frac{1}{2}}  \left(\int_0^T \abs{g_{kh}(t,\gamma_k(t))}^2\,dt\right)^{\frac{1}{2}}\\
&\le \left(\int_0^T \abs{v_{kh}(t,\gamma_k(t))}^2\,dt\right)^{\frac{1}{2}} \, \norm{g_{kh}}_{L^2(I;L^\infty(\Omega))}\\
&\le c \lh^{\frac{1}{2}} \left(\int_0^T \abs{v_{kh}(t,\gamma_k(t))}^2\,dt\right)^{\frac{1}{2}} \, \norm{\nabla g_{kh}}_{L^2(I\times\Om)}.
\end{aligned}
\]
This gives the desired estimate.
\end{proof}

We proceed with the proof of Theorem~\ref{thm: global best approx}. From Lemma~\ref{lemma 1}, Lemma~\ref{lemma 2}, and Lemma~\ref{lemma 3}. It follows that
\begin{multline*}
\int_0^T\left(\|\sigma_k\Delta_h g_{kh}\|^2_{L^2(\Om)}+\|\na g_{kh}\|^2_{L^2(\Om)}\right)dt\le C_\eps\lh\int_0^T|v_{kh}(t,\gamma_k(t))|^2dt \\+C\eps\int_0^T\|\sigma_k\Delta_h g_{kh}\|^2_{L^2(\Om)}dt.
\end{multline*}
Taking $\eps$ sufficiently small we have \eqref{eq: needed estimate for greens}.
From \eqref{eq: u_kh plus others} we can conclude that
\begin{equation*}
\int_0^T|v_{kh}(t,\gamma_k(t))|^2dt\le C\lh^2
\left(\|v\|^2_{L^2(I;L^\infty(\Om))}+h^{-\frac{4}{p}}\|\pi_k v\|^2_{L^2(I; L^p(\Om))}\right),
 \end{equation*}
 for some constant $C$ independent of $h$ and $k$.
Using that dG(0)cG(1) method is invariant on $\Xkh$, by replacing $v$ and $v_{kh}$ with $v-\chi$ and $v_{kh}-\chi$ for any $\chi\in \Xkh$,
 we obtain
\begin{equation*}
\int_0^T|(v_{kh}-\chi)(t,\gamma_k(t))|^2dt\le C\lh^2
\left(\|v-\chi\|^2_{L^2(I;L^\infty(\Om))}+h^{-\frac{4}{p}}\|\pi_k v-\chi\|^2_{L^2(I; L^p(\Om))}\right).
 \end{equation*}
By the triangle inequality and the above estimate we deduce
\begin{align*}
\int_0^T|(v-v_{kh})(t,\gamma_k(t))|^2dt&\le \int_0^T|(v_{kh}-\chi)(t,\gamma_k(t))|^2dt +\int_0^T|(v-\chi)(t,\gamma_k(t))|^2dt\\
&\le C\lh^2
\left(\|v-\chi\|^2_{L^2(I;L^\infty(\Om))}+h^{-\frac{4}{p}}\|\pi_k v-\chi\|^2_{L^2(I; L^p(\Om))}\right).
 \end{align*}
Taking the infimum over $\chi$, we obtain Theorem~\ref{thm: global best approx}.
\end{proof}

\subsection{Interior error estimate}

To obtain optimal error estimates we will also require the following interior result.

\begin{theorem}[Interior approximation] \label{thm: apriori local}
Let $B_{d,m}:=B_d(\gamma(t_m))$ denote a ball of radius $d$ centered at $\gamma(t_m)$.
Assume $v$ and $v_{kh}$ satisfy \eqref{eq: heat equation} and \eqref{eq: discrete heat with RHS} respectively and let  $d>4h$. Then there exists a constant $C$ independent of $h$, $k$ and $d$ such that for any $1\le p\le \infty$
\begin{multline}\label{eq:local_v_vkh}
\int_{0}^T|(v-v_{kh})(t,\gamma_k(t))|^2dt\\ \le C \lh^2\inf_{\chi \in \Xkh}\Biggl\{\sum_{m=1}^M\left(\|v-\chi\|^2_{L^2(I_m;L^\infty(B_{d,m}))}+h^{-\frac{4}{p}}\|\pi_{k}v-\chi\|^2_{L^2(I_m;L^p(B_{d,m}))}\right) \\
 +d^{-2} \left(\norm{v-\chi}^2_{L^2(I;L^2(\Om))}+\norm{\pi_{k}v-\chi}^2_{L^2(I;L^2(\Om))}+h^2\norm{\na(v-\chi)}^2_{L^2(I;L^2(\Om))}\right)\Biggr\}.
\end{multline}
\end{theorem}
\begin{proof}
To obtain the interior estimate we introduce a smooth cut-off function $\omega$ in space and piecewise constant in time, such that $\omega_m:=\omega\rvert_{I_m}$,
\begin{subequations} \label{def: properties of omega a local}
\begin{align}
\omega_m(x)&\equiv 1,\quad x\in B_{d/2,m} \label{eq: property 1 of omega_a local}\\
\omega_m(x)&\equiv 0,\quad x\in \Om\setminus B_{d,m} \label{eq: property 2 of omega_a local}\\
 |\na \omega_m|&\le Cd^{-1}, \quad |\na^2 \omega_m|\le Cd^{-2}, \label{eq: property 3 of omega_a local}.
\end{align}
\end{subequations}
As in the proof of Theorem~\ref{thm: global best approx} we obtain by~\eqref{eq: starting with B} that
\begin{equation}\label{eq: local starting expression}
\int_0^T|v_{kh}(t,\gamma_k(t))|^2dt =B(v_{kh}, g_{kh}) =B(v, g_{kh})=B(\om v, g_{kh})+B((1-\om)v, g_{kh}),
\end{equation}
where $g_{kh}$ is the solution of~\eqref{eq: discrete heat with dirac RHS}. Note that $\omega v$ is discontinuous in time. The first term can be estimated using the global result from Theorem~\ref{thm: global best approx}. To this end we introduce the solution $\tilde v_{kh} \in \Xkh$ defined by
\[
B(\tilde v_{kh} - \om v, \varphi_{kh}) = 0 \quad \text{for all }  \varphi_{kh} \in \Xkh.
\]
There holds
\[
\begin{aligned}
B(\om v, g_{kh}) &= B(\tilde v_{kh}, g_{kh}) = \int_0^T v_{kh}(t,\gamma_k(t)) \tilde v_{kh}(t,\gamma_k(t))\, dt\\
&\le \half \int_0^T \abs{v_{kh}(t,\gamma_k(t))}^2\, dt + \half \int_0^T \abs{\tilde v_{kh}(t,\gamma_k(t))}^2\, dt.
\end{aligned}
\]
 Applying Theorem~\ref{thm: global best approx} for the second term, we have
\[
\begin{aligned}
\int_0^T |\tilde v_{kh}&(t,\gamma_k(t))|^2 dt \le  C \lh^2
\left(\|\om v\|^2_{L^2(I;L^\infty(\Om))}+h^{-\frac{4}{p}}\|\pi_k (\om v)\|^2_{L^2(I; L^p(\Om))}\right)\\
&\le  C \lh^2
\sum_{m=1}^M\left(\|v\|^2_{L^2(I_m;L^\infty(B_{d,m}))}+h^{-\frac{4}{p}}\|\pi_k v\|^2_{L^2(I_m; L^p(B_{d,m}))}\right).
\end{aligned}
\]
From \eqref{eq: local starting expression}, canceling $\half \int_0^T \abs{v_{kh}(t,\gamma_k(t))}^2\, dt$ and using the above estimate, we obtain
\begin{equation}\label{eq:local_after_first_step}
\begin{aligned}
\int_0^T&|v_{kh}(t,\gamma_k(t))|^2dt \le B((1-\om)v, g_{kh})\\
&+C \lh^2  \sum_{m=1}^M
\left(\|v\|^2_{L^2(I_m;L^\infty(B_{d,m}))}+h^{-\frac{4}{p}}\|\pi_k v\|^2_{L^2(I_m; L^p(B_{d,m}))}\right).
\end{aligned}
\end{equation}
It remains to estimate the term $B((1-\om)v, g_{kh})$. Using the dual expression~\eqref{eq:B_dual} of the bilinear form $B$ we obtain
\begin{equation}\label{eq:J1_J2_local}
\begin{aligned}
B((1-\om)v, g_{kh}) &= \sum_{m=1}^M\bigg((\na((1-\om_m)v), \na g_{kh})_{\ImOm} - ((1-\om_m)v_m, [g_{kh}]_m)_{\Om} \bigg)\\
&= J_1 + J_2.
\end{aligned}
\end{equation}
To estimate $J_1$ we define $\psi = (1-\om)v$ and proceed using
the Ritz projection $R_h \colon H^1_0(\Omega) \to V_h$ defined by
\begin{equation}\label{eq:Ritz_proj}
(\nabla R_hv,\nabla \chi)_{\Om} = (\nabla v,\nabla \chi)_{\Om}, \quad \forall \chi\in V_h.
\end{equation}
There holds
\[
\begin{aligned}
(\na\psi, \na g_{kh})_{\ImOm}&=(\na R_h\psi, \na g_{kh})_{\ImOm}=-(R_h\psi, \De_h g_{kh})_{\ImOm}\\
&=-(R_h\psi, \De_h g_{kh})_{I_m\times B_{d/4,m}}-(R_h\psi, \De_h g_{kh})_{I_m\times\Om\setminus B_{d/4,m}}\\
&\le \|R_h\psi\|_{L^2(I_m;L^\infty(B_{d/4,m}))} \|\De_h g_{kh}\|_{L^2(I_m;L^1(B_{d/4,m}))}\\
&\ +\|\si^{-1}_{k,m}R_h\psi\|_{L^2(I_m\times\Om\setminus B_{d/4,m})}\|\si_{k,m}\De_h g_{kh}\|_{L^2(I_m\times\Om)}.
\end{aligned}
\]
Using the estimate
\[
\begin{aligned}
\|\De_h g_{kh}\|_{L^2(I_m;L^1(B_{d/4,m}))} &\le \ltwonorm{\sigma^{-1}_{k,m}} \|\si_{k,m}\De_h g_{kh}\|_{L^2(I_m\times B_{d/4,m})} \\
&\le C \lh^\half \|\sigma_{k,m}\De_h g_{kh}\|_{L^2(I_m\times\Om)},
\end{aligned}
\]
where in the last step we used \eqref{eq: property 1 of sigma_k}, 
we obtain
\begin{equation}\label{eq:est_psi_g_kh}
\begin{aligned}
(\na\psi, \na g_{kh})_{\ImOm} &\le C \lh^\half \left( \|R_h\psi\|_{L^2(I_m;L^\infty(B_{d/4,m}))} + \|\si^{-1}_{k,m}R_h\psi\|_{L^2(I_m\times\Om\setminus B_{d/4,m})}\right)\\
&\quad\times\|\si_{k,m}\De_h g_{kh}\|_{L^2(I_m\times\Om)}.
\end{aligned}
\end{equation}
By the interior pointwise error estimates from Theorem 5.1 in \cite{SchatzAH_WahlbinLB_1977}, we have for each $t\in I_m$,
$$
\begin{aligned}
\|R_h\psi(t)\|_{L^\infty(B_{d/4,m})} &\le c\lh\|\psi(t)\|_{L^\infty(B_{d/2,m})}+Cd^{-1}\|R_h\psi(t)\|_{L^2(B_{d/2,m})}\\
&=Cd^{-1}\|R_h\psi(t)\|_{L^2(B_{d/2,m})},
\end{aligned}
$$
since the support of $\psi_m=(1-\om_m)v$ is contained  in $\Om\setminus B_{d/2,m}$. On $\Om\setminus B_{d/4,m}$ there holds $\sigma_{k,m} \ge d/4$ and therefore for each $t\in I_m$,
\[
\|\si^{-1}_{k,m}R_h\psi(t)\|_{L^2(\Om\setminus B_{d/4,m})} \le Cd^{-1} \|R_h\psi(t)\|_{L^2(\Om\setminus B_{d/4,m})}.
\]
Inserting the last two estimates into~\eqref{eq:est_psi_g_kh} we get
\[
(\na\psi, \na g_{kh})_{I_m\times\Om} \le C d^{-1} \lh^\half \|R_h \psi\|_{L^2(I_m\times\Om)} \|\sigma_{k,m}\De_h g_{kh}\|_{L^2(I_m\times\Om)}.
\]
Using a standard elliptic estimate and recalling $\psi = (1-\om) v$ we have
\[
\begin{aligned}
\ltwonorm{R_h \psi(t)} &\le \ltwonorm{\psi(t)} + \ltwonorm{\psi(t)-R_h\psi(t)}\\
&\le  \ltwonorm{\psi(t)} + c h \ltwonorm{\nabla \psi(t)}\\
&\le \ltwonorm{v(t)} + c h \ltwonorm{(1-\om(t))\nabla v(t) - \nabla \om(t) v(t)}\\
&\le c \ltwonorm{v(t)} + ch \ltwonorm{\nabla v(t)},
\end{aligned}
\]
where in the last step we used $\abs{\nabla \om(t)} \le c d^{-1} \le c h^{-1}$. This results in
\[
(\na\psi, \na g_{kh})_{I_m\times\Om} \le C d^{-1} \lh^\half \left(\|v\|_{L^2(I_m\times\Om)} + h \|\nabla v\|_{L^2(I_m\times\Om)} \right) \|\sigma_{k,m}\De_h g_{kh}\|_{L^2(I_m\times\Om)}.
\]
Therefore, we get
\begin{equation}\label{eq:J1_local}
J_1 \le c d^{-1} \lh^\half \left(\norm{v}_{L^2(I;L^2(\Om))} + h \norm{\nabla v}_{L^2(I;L^2(\Om))} \right) \norm{\sigma_k\De_h g_{kh}}_{L^2(I;L^2(\Om))}.
\end{equation}
For $J_2$ we obtain
\begin{equation}\label{eq:J2_local}
\begin{aligned}
J_2&\le \sum_{m=1}^M\ltwonorm{\si_m^{-1}(1-\om_m)v_m} k_m^\half k_m^{-\half}\ltwonorm{\si_{k,m} [g_{kh}]_m}\\
&\le C\left(\sum_{m=1}^M d^{-2} k_m \ltwonorm{(1-\om_m)v_m}^2\right)^{\half}\left(\sum_{m=1}^Mk_m^{-1}\ltwonorm{\si_{k,m} [g_{kh}]_m}^2\right)^{\half}\\
&\le C d^{-1} \norm{\pi_k v}_{L^2(I;L^2(\Om))}\left(\sum_{m=1}^Mk_m^{-1}\ltwonorm{\si_{k,m} [g_{kh}]_m}^2\right)^{\half},
\end{aligned}
\end{equation}
where we used that $supp (1-\om_m)v_m \subset \Omega \setminus B_{d/2,m}$ and $\sigma_{k,m} \ge d/2$ on this set as well as the definition of $\pi_k$~\eqref{eq:pi_k}.
Inserting the estimate~\eqref{eq:J1_local} for $J_1$ and the estimate~\eqref{eq:J2_local} for $J_2$ into~\eqref{eq:J1_J2_local} we obtain
\begin{align*}
B((1-\om)v, g_{kh})&\le C d^{-1} \lh^\half
\left(\sum_{m=1}^M\norm{\sigma_{k,m}\De_h g_{kh}}_{L^2(I_m\times\Om)}^2 + k_m^{-1}\ltwonorm{\si_{k,m} [g_{kh}]_m}^2\right)^\half\\
&\quad \times \left(\norm{v}_{L^2(I;L^2(\Om))} + h \norm{\nabla v}_{L^2(I;L^2(\Om))} + \norm{\pi_k v}_{L^2(I;L^2(\Om))} \right).
\end{align*}
Using the estimate \eqref{eq: needed estimate for greens} and Lemma \ref{lemma 2}
\begin{align*}
B((1-\om)v, g_{kh})&\le C d^{-1}  \lh  \left( \int_0^T |v_{kh}(t,\gamma_k(t))|^2\,dt\right)^\half\\
&\quad \times \left(\norm{v}_{L^2(I;L^2(\Om))} + h \norm{\nabla v}_{L^2(I;L^2(\Om))} + \norm{\pi_k v}_{L^2(I;L^2(\Om))} \right).
\end{align*}
Inserting this inequality into~\eqref{eq:local_after_first_step} we obtain
\begin{multline*}
\int_0^T|v_{kh}(t,\gamma_k(t))|^2dt \le C  \lh^2
\left(\sum_{m=1}^M\|v\|^2_{L^2(I_m;L^\infty(B_{d,m}))}+h^{-\frac{4}{p}}\|\pi_k v\|^2_{L^2(I_m; L^p(B_{d,m}))}\right) \\
+C d^{-2}  \lh^2 \left(\norm{v}_{L^2(I;L^2(\Om))}^2 + h^2 \norm{\nabla v}^2_{L^2(I;L^2(\Om))} + \norm{\pi_k v}^2_{L^2(I;L^2(\Om))} \right).
\end{multline*}
Using that the dG($0$)cG($1$) method is invariant on $\Xkh$, by replacing $v$ and $v_{kh}$ with $v-\chi$ and $v_{kh}-\chi$ for any $\chi\in \Xkh$, we obtain the estimate in  Theorem~\ref{thm: apriori local}.
\end{proof}

\section{Discretization of the optimal control problem}\label{sec: optimal control}
In this section we  describe the discretization of the optimal control problem \eqref{eq:intro-obj}-\eqref{eq:intro-state} and prove our main result, Theorem~\ref{thm: main}. We start with discretization of the state equation.  For a given control $q \in Q$ we define the corresponding discrete state $u_{kh} = u_{kh}(q) \in \Xkh$ by
\begin{equation}\label{eq:state_kh}
B(u_{kh},\varphi_{kh}) = \int_0^T q(t)\varphi_{kh}(t,\gamma_k(t))\, dt \quad \text{for all }\; \varphi_{kh} \in \Xkh.
\end{equation}
Using the weak formulation for $u=u(q)$ from Proposition~\ref{prop:reg_u2} we obtain the perturbed Galerkin orthogonality,
\begin{equation}\label{eq: error u and u_kh}
B(u-u_{kh},\varphi_{kh}) =\int_0^T q(t)\left(\varphi_{kh}(t,\gamma(t))-\varphi_{kh}(t,\gamma_k(t))\right)\, dt  \quad \text{for all }\; \varphi_{kh} \in \Xkh.
\end{equation}
Note, that the jump terms involving $u$ vanish due to the fact that 
$$
u \in H^1(I;W^{-1,s}(\Omega))\hookrightarrow C(I;W^{-1,s}(\Omega))
$$
 and $\varphi_{kh,m} \in W^{1,\infty}(\Omega)$.

Similarly to the continuous problem, we define the discrete reduced cost functional $j_{kh} \colon Q \to \mathbb{R}$ by
\[
j_{kh}(q) = J(q,u_{kh}(q)),
\]
where $J$ is the cost function in~\eqref{eq:intro-obj}. The discretized optimal control problem is then given as
\begin{equation}\label{eq:min_pr_kh}
\min\,  j_{kh}(q), \quad q \in \Qad,
\end{equation}
where $\Qad$ is the set of admissible controls~\eqref{eq:qad}. We note, that the control variable $q$ is not explicitly discretized, cf.~\cite{HinzeM_2005a}. With standard arguments one proves the existence of a unique solution $\oq_{kh} \in \Qad$ of~\eqref{eq:min_pr_kh}. Due to convexity of the problem, the following condition is necessary and sufficient for the optimality,
\begin{equation}\label{eq:opt_cond_kh}
j'_{kh}(\oq_{kh})(\partial q - \oq_{kh}) \ge 0 \quad \text{for all }\; \partial q \in \Qad.
\end{equation}
As on the continuous level, the directional derivative $j'_{kh}(q)(\partial q)$ for given $q,\partial q \in Q$ can be expressed as
\[
j'_{kh}(q)(\partial q) = \int_I \left(\alpha q(t) + z_{kh}(t,\gamma_k(t)) \right) \partial q(t)\, dt,
\]
where $z_{kh} = z_{kh}(q)$ is the solution of the discrete adjoint equation
\begin{equation}\label{eq:adjoint_kh}
B(\varphi_{kh},z_{kh}) = (u_{kh}(q) - \hat{u},\varphi_{kh})_{I\times \Om} \quad \text{for all }\; \varphi_{kh} \in \Xkh.
\end{equation}
The discrete adjoint state, which corresponds to the discrete optimal control $\oq_{kh}$ is denoted by $\oz_{kh} = z(\oq_{kh})$. The variational inequality~\eqref{eq:opt_cond_kh} is  equivalent to the following pointwise projection formula, cf.~\eqref{eq:proj},
\[
\oq_{kh}(t) = P_{\Qad}\left(-\frac{1}{\alpha} \oz_{kh}(t,\gamma_k(t))\right),
\]
or
\[
\oq_{kh,m} = P_{\Qad}\left(-\frac{1}{\alpha} \oz_{kh,m}(\gamma_{k,m})\right),
\]
on each $I_m$.
Due to the fact that $\oz_{kh} \in \Xkh$, we have $\oz_{kh}(t,\gamma_k(t))$ is piecewise constant and therefore by the projection formula also $\oq_{kh}$ is piecewise constant. As a result no explicit discretization of the control variable is required. 

To prove Theorem~\ref{thm: main} we first need estimates for the error in the state and in the adjoint variables for a given (fixed) control $q$. Due to the structure of the optimality conditions, we will have to estimate the error $\norm{z(\cdot,\gamma(\cdot))-z_{kh}(\cdot,\gamma_k(\cdot))}_I$, where $z=z(q)$ and $z_{kh}=z_{kh}(q)$.
Note, that $z_{kh}$ is not the Galerkin projection of $z$ due to the fact that the right-hand side of the adjoint equation~\eqref{eq:intro-adjoint} involves $u=u(q)$ and the right-hand side of the discrete adjoint equation~\eqref{eq:adjoint_kh} involves $u_{kh}=u_{kh}(q)$. To obtain an estimate of optimal order,
 we will first estimate the error $u - u_{kh}$ with respect to the $L^2(I;L^1(\Omega))$ norm. Note, that an $L^2$ estimate would not lead to an optimal result.

\begin{theorem}\label{th:est_u_l2l1}
Let $q\in Q$ be given and let $u=u(q)$  be the solution of the state equation~\eqref{eq:intro-state} and $u_{kh}=u_{kh}(q)\in \Xkh$ be the solution of the discrete state equation~\eqref{eq:state_kh}. Then there holds the following estimate
\[
\norm{u - u_{kh}}_{L^2(I;L^1(\Omega))} \le \left( C \lh^2 (k + h^2)+C_\gamma|\ln{h}|k\right) \norm{q}_I.
\]
\end{theorem}

\begin{proof}
We denote by $e=u-u_{kh}$ the error and consider the following auxiliary dual problem
\[
\begin{aligned}
-w_t(t,x)-\Delta w(t,x) &=  b(t,x), & (t,x) &\in \IOm,\;  \\
    w(t,x) &= 0,    & (t,x) &\in I\times\pa\Omega, \\
   w(T,x) &= 0,    & x &\in \Omega,
\end{aligned}
\]
where 
$$
b(t,x) = \sgn(e(t,x))\norm{e(t,\cdot)}_{L^1(\Om)}\in L^2(I;L^\infty(\Om))
$$
and the corresponding discrete solution $w_{kh} \in \Xkh$ defined by
\[
B(\varphi_{kh},w-w_{kh})=0, \quad \forall \varphi_{kh}\in \Xkh.
\]
Using \eqref{eq: error u and u_kh} for $e=u-u_{kh}$ and the Galerkin orthogonality for  $w-w_{kh}$ we obtain,
\begin{equation}\label{eq:err_u_l2l1}
\begin{aligned}
\int_0^T &\norm{e(t,\cdot)}_{L^1(\Omega)}^2\,dt = (e,\sgn(e)\norm{e(t,\cdot)}_{L^1(\Om)})_{I \times \Omega} \\
&= (e,b)_{I \times \Omega}\\
&=B(e, w)\\
&=B(e, w-w_{kh})+B(e, w_{kh})\\
&=B(u, w-w_{kh})+B(e, w_{kh})\\
&=\int_0^T q(t)(w-w_{kh})(t,\gamma(t))dt+\int_0^T q(t)(w_{kh}(t,\gamma(t))-w_{kh}(t,\gamma_k(t)))dt\\
&=\int_0^T q(t)(w(t,\gamma(t))-w_{kh}(t,\gamma_k(t)))dt\\
&\le \norm{q}_I\, \left(\int_0^T|w(t,\gamma(t))-w_{kh}(t,\gamma_k(t))|^2dt\right)^{\frac{1}{2}}\\
&\le\norm{q}_I\, \left(\int_0^T\left(|w(t,\gamma(t))-w(t,\gamma_k(t))|^2+|(w-w_{kh})(t,\gamma_k(t))|^2\right)dt\right)^{\frac{1}{2}}.
\end{aligned}
\end{equation}
Using the local estimate from Theorem \ref{thm: apriori local} with $B_{d,m}\subset \Om_1$ for any $m=1,\dots,M,$  where $\Om_0\subset\subset\Om_1\subset\subset\Om$, we obtain
\[
\begin{aligned}
\int_0^T|(w&-w_{kh})(t,\gamma_k(t))|^2dt\\\le & C\lh^2\int_0^T\left(\|w-\chi\|^2_{L^\infty(\Om_1)}+h^{-\frac{4}{p}}\|\pi_k w -\chi\|^2_{L^p(\Om_1)}\right)dt\\ &+C\lh^2\int_{0}^T\left(\|w-\chi\|^2_{L^2(\Omega)}+h^2\|\na(w-\chi)\|^2_{L^2(\Omega)} +\|\pi_k w - \chi\|^2\right)dt\\
=& J_1+J_2+J_3+J_4+J_5.
\end{aligned}
\]
We take $\chi=i_h\pi_k w$, where $i_h$ is the modified Cl\'{e}ment interpolant and $\pi_k$ is the projection defined in \eqref{eq:pi_k}. Thus,  by the triangle inequality, approximation theory, inverse inequality and the stability of the Cl\'{e}ment interpolant in $L^p$ norm, we have
\begin{align*}
J_1&\le C\lh^2\int_{0}^T\left(\|w-i_h w\|^2_{L^\infty(\Om_1)}+\|i_h(w-\pi_k w)\|^2_{L^\infty(\Om_1)}\right)dt \\
& \le C\lh^2\int_{0}^T \left(h^{4-\frac{4}{p}}\|w\|^2_{W^{2,p}(\Om_1)}+h^{-\frac{4}{p}}\|i_h(w-\pi_k w)\|^2_{L^p(\Om_1)}\right)dt\\
&\le Ch^{-\frac{4}{p}}\lh^2(h^4+k^2)\int_0^T\left(\| w\|^2_{W^{2,p}(\Om_1)}+\|w_t\|^2_{L^p(\Om_1)}\right)dt.
\end{align*}
$J_2$ can be estimated similarly since for $\chi=i_h\pi_k w$ by the triangle inequality we have
\[
\|\pi_{k}w-i_h\pi_k w\|_{L^p(\Om)}\le \|\pi_{k}w-w\|_{L^p(\Om)}+\|w-i_hw\|_{L^p(\Om)}+\|i_{h}(w-\pi_k w)\|_{L^p(\Om)}.
\]
As a result
\[
J_1+J_2
\le Ch^{-\frac{4}{p}}\lh^2(h^4+k^2)\int_0^T\left(\| w\|^2_{W^{2,p}(\Om_1)}+\|w_t\|^2_{L^p(\Om_1)}\right)dt.
\]
Using Lemma~\ref{lemma: regularity local_Om_0}, we obtain
\begin{equation}\label{eq: follows from Lemma 2}
\int_{0}^T\left(\| w\|^2_{W^{2,p}(\Om_1)}+\|w_t\|^2_{L^p(\Om_1)}\right)dt \le C p^2\norm{b}_{L^2(I;L^p(\Omega))}^2 \le Cp^2 \norm{e}_{L^2(I;L^1(\Omega))}^2,
\end{equation}
and hence
\begin{equation}\label{eq: J1 and J2}
J_1+J_2
\le Ch^{-\frac{4}{p}}\lh^2(h^4+k^2)p^2 \norm{e}_{L^2(I;L^1(\Omega))}^2.
\end{equation}
For the terms $J_3$ and $J_4$ we obtain using an $L^2$-estimate from~\cite{MeidnerD_VexlerB_2008a}
\[
\begin{aligned}
J_3+J_4 &\le C\lh^2 (h^4+k^2)\left(\norm{\nabla^2 w}^2_{L^2(I;L^2(\Omega))} + \norm{w_t}^2_{L^2(I;L^2(\Omega))}\right)\\
& \le C   \lh^2(h^4+k^2) \norm{b}^2_{L^2(I;L^2(\Omega))}\\
&\le C   \lh^2(h^4+k^2) \norm{e}^2_{L^2(I;L^1(\Omega))}.
\end{aligned}
\]
$J_5$ can be estimated similarly since by the triangle inequality
$$
\|\pi_kw-i_h\pi_kw \|_{L^2(I\times\Om)}\le \|\pi_kw-w \|_{L^2(I\times\Om)}+\|w-i_h\pi_kw \|_{L^2(I\times\Om)}.
$$
On the other hand using that $w\in L^2(I;W^{2,p}(\Om_0))$ for $p>2$ and that  $W^{2,p}(\Om_0))\hookrightarrow C^1(\Om_0)$ for $p>2$, and using Assumption 1, we have
$$
\begin{aligned}
\int_0^T|w(t,\gamma(t))-w(t,\gamma_k(t))|^2dt&\le 
\int_0^T\|w(t,\cdot)\|^2_{C^1(\Om_0)}|\gamma(t)-\gamma_k(t)|^{2}dt\\
&\le C\|\gamma-\gamma_k\|^{2}_{C^0(I)}\int_0^T\|w(t,\cdot)\|^2_{W^{2,p}(\Om_0)}dt\\
&\le CC_\gamma^2 k^{2}\|w\|^2_{L^2(I;W^{2,p}(\Om_0))}\\
&\le CC_\gamma^2k^{2} p^2\norm{b}^2_{L^2(I;L^p(\Omega))}\\
&\le CC_\gamma^2k^{2}p^2 \norm{e}^2_{L^2(I;L^1(\Omega))},
\end{aligned}
$$
where in the last two steps we used \eqref{eq: follows from Lemma 2}.
Combining the estimate for $J_1$, $J_2$, $J_3$, $J_4$, $J_5$ and the above estimate and inserting them into~\eqref{eq:err_u_l2l1} we obtain:
\[
\norm{e}_{L^2(I;L^1(\Omega))} \le \left(C\lh(p h^{-\frac{2}{p}} +1)(h^2+k)+C_\gamma pk\right)\|q\|_{L^2(I)}.
\]
Setting $p = \lh$ completes the proof.
\end{proof}

In the following theorem we provide an estimate of the error in the adjoint state for fixed control $q$.
\begin{theorem}\label{th:est_z}
Let $q\in Q$ be given and let $z=z(q)$  be the solution of the adjoint equation~\eqref{eq:intro-adjoint} and $z_{kh}=z_{kh}(q)\in \Xkh$ be the solution of the discrete adjoint equation~\eqref{eq:adjoint_kh}. Then there holds the following estimate
\begin{align*}
\bigg(\int_0^T \abs{z(t,\gamma(t)) &- z_{kh}(t,\gamma_k(t))}^2\,dt\bigg)^{\frac{1}{2}} \\
&\le C \left(\lh^3 (k + h^2)+C_\gamma\lh k\right) 
\left(\norm{q}_{L^2(I)}+\norm{\hat u}_{L^2(I;L^\infty(\Omega))}\right).
\end{align*}
\end{theorem}
\begin{proof}
First by the triangle inequality
$$
\begin{aligned}
\int_0^T \abs{z(t,\gamma(t)) - z_{kh}(t,\gamma_k(t))}^2\,dt\le &\int_0^T \abs{z(t,\gamma(t)) - z(t,\gamma_k(t))}^2\,dt\\
&+\int_0^T \abs{(z - z_{kh})(t,\gamma_k(t))}^2\,dt.
\end{aligned}
$$
Using Proposition \ref{prop:reg_z} and the assumptions on $\gamma$, we have similarly to Theorem \ref{th:est_u_l2l1}
\begin{align*}
\int_0^T \abs{z(t,\gamma(t)) - z(t,\gamma_k(t))}^2\,dt&\le C\|\gamma-\gamma_k\|^2_{C^0(I)} \int_0^T \|z(t,\cdot)\|^2_{C^{1}(\Om_0)}\,dt\\
&\le CC_\gamma^2 k^{2} \int_0^T \|z(t,\cdot)\|^2_{W^{2,p}(\Om_0)}\,dt\\
&\le CC_\gamma^2 pk^{2} \left(\|q\|^2_{L^2(I)}+\|\hat{u}\|^2_{L^2(I\times\Om)}\right).
\end{align*}
Setting $p=\lh$, we obtain
\begin{equation}\label{eq: error z gamma - z gamma_k}
    \left(\int_0^T \abs{z(t,\gamma(t)) - z(t,\gamma_k(t))}^2\,dt\right)^{\frac{1}{2}}\le CC_\gamma \lh k \left(\|q\|_{L^2(I)}+\|\hat{u}\|_{L^2(I\times\Om)}\right).
\end{equation}

Next, we introduce an intermediate adjoint state $\widetilde z_{kh} \in \Xkh$ defined by
\[
B(\varphi_{kh},\widetilde z_{kh}) = (u - \hat u,\varphi_{kh}) \quad \text{for all }\; \varphi_{kh} \in \Xkh,
\]
where $u = u(q)$ and therefore $\widetilde z_{kh}$ is the Galerkin projection of $z$. By the local best approximation result of Theorem~\ref{thm: apriori local} for any $\chi\in \Xkh$ we have
\[
 \begin{aligned}
\int_{0}^T|(z-\widetilde z_{kh})&(t,\gamma_k(t))|^2\,dt\le C\lh^2\int_{0}^T\left(\|z-\chi\|^2_{L^\infty(\Om_1)}+h^{-\frac{4}{p}}\|\pi_k z-\chi\|^2_{L^p(\Om_1)}\right)dt\\ &+C\lh^2\int_0^T\left(\|z-\chi\|^2_{L^2(\Omega)}+h\|\na(z-\chi)\|^2_{L^2(\Omega)}+\|\pi_kz-\chi\|^2_{L^2(\Omega)}\right)dt\\
&=J_1+J_2+J_3+J_4+J_5.
\end{aligned}
\]
The terms $J_1$, $J_2$, $J_3$, $J_4$ and $J_5$ can be estimated the same way as in the proof of Theorem~\ref{th:est_u_l2l1} using the regularity result for the adjoint state $z$ from Proposition~\ref{prop:reg_z}. This results in
\[
\int_0^T|(z-\widetilde z_{kh})(t,\gamma_k(t)|^2\,dt \le C\lh^2 (p h^{-\frac{2}{p}} +1)^2(h^2+k)^2
\left(\norm{q}^2_{L^2(I)}+\norm{\hat u}^2_{L^2(I;L^\infty(\Omega))}\right).
\]
Setting $p = \lh$ and taking square root, we obtain
\begin{equation}\label{eq:interm_z}
\left(\int_{0}^T|(z-\widetilde z_{kh})(t,\gamma_k(t))|^2\,dt\right)^{\frac{1}{2}} \le C\lh^2(h^2+k)\left(\norm{q}_{L^2(I)}+\norm{\hat u}_{L^2(I;L^\infty(\Omega))}\right).
\end{equation}
It remains to estimate the corresponding error between $\widetilde z_{kh}$ and $z_{kh}$. We denote $e_{kh} = \widetilde z_{kh} - z_{kh} \in \Xkh$. Then we have
\[
B(\varphi_{kh},e_{kh}) = (u - u_{kh}, \varphi_{kh})_{I\times \Om}  \quad \text{for all }\; \varphi \in \Xkh.
\]
As in the proof of Lemma~\ref{lemma 3} we use the fact that
\[
\norm{\nabla v}^2_{L^2(I \times \Om)} \le B(v,v)
\]
holds for all $v\in \Xkh$. 
Applying this inequality together with the discrete Sobolev inequality, see~\cite{BrennerSC_ScottLR_2008}, results in
\[
\begin{aligned}
\norm{e_{kh}}^2_{L^2(I;L^\infty(\Omega))}&\le C\lh\norm{\nabla e_{kh}}_{L^2(I \times \Om)}^2\\
 &\le  C\lh B(e_{kh},e_{kh}) \\
& =  C\lh (u - u_{kh}, e_{kh})_{I\times \Om}\\
& \le  C\lh\norm{u - u_{kh}}_{L^2(I;L^1(\Omega))} \norm{e_{kh}}_{L^2(I;L^\infty(\Omega))}.
\end{aligned}
\]
Therefore 
\[
\norm{e_{kh}}_{L^2(I;L^\infty(\Omega))} \le  C \lh \norm{u - u_{kh}}_{L^2(I;L^1(\Omega))}.
\]
Using Theorem~\ref{th:est_u_l2l1} we obtain
\[
\norm{e_{kh}}_{L^2(I;L^\infty(\Omega))} \le C \left(\lh^3(k + h^2)+C_\gamma\lh  k\right) \norm{q}_{L^2(I)}.
\]
Combining this estimate with~\eqref{eq:interm_z} we complete the proof.
\end{proof}


Using the result of Theorem~\ref{th:est_z} we proceed with the proof of Theorem~\ref{thm: main}.
\begin{proof}
Due to the quadratic structure of discrete reduced functional $j_{kh}$ the second derivative $j''_{kh}(q)(p,p)$ is independent of $q$ and there holds
\begin{equation}\label{eq:q_coerc}
j''_{kh}(q)(p,p) \ge \alpha \norm{p}^2_{L^2(I)} \quad \text{for all }\; p \in Q.
\end{equation}
Using optimality conditions~\eqref{eq:opt_cond} for $\oq$ and~\eqref{eq:opt_cond_kh} for $\oq_{kh}$ and the fact that $\oq,\oq_{kh} \in \Qad$ we obtain
\[
-j'_{kh}(\oq_{kh})(\oq-\oq_{kh}) \le 0 \le - j'(\oq)(\oq-\oq_{kh}).
\]
Using the coercivity~\eqref{eq:q_coerc} we get
\[
\begin{aligned}
\alpha \norm{\oq-\oq_{kh}}_{L^2(I)}^2 & \le j''_{kh}(\oq)(\oq-\oq_{kh},\oq-\oq_{kh})_I\\
& = j'_{kh}(\oq)(\oq-\oq_{kh}) - j'_{kh}(\oq_{kh})(\oq-\oq_{kh})\\
&\le j'_{kh}(\oq)(\oq-\oq_{kh}) - j'(\oq)(\oq-\oq_{kh}) \\
&=(z(\oq)(t,\gamma(t))-z_{kh}(\oq)(t,\gamma_k(t)),\oq-\oq_{kh})_I\\
& \le \left(\int_0^T \abs{z(\oq)(t,\gamma(t)) - z_{kh}(\oq)(t,\gamma_k(t))}^2\,dt\right)^{\frac{1}{2}}  \norm{\oq-\oq_{kh}}_{L^2(I)}.
\end{aligned}
\]
Applying Theorem~\ref{th:est_z} completes the proof.
\end{proof}

\bibliography{curve}
\bibliographystyle{siam}
\end{document}